\documentclass{article}
\usepackage[utf8]{inputenc}
\usepackage{amsthm}
\usepackage{amsmath}
\usepackage{amssymb}
\usepackage{graphicx}
\usepackage{mathtools}

\newtheorem{theorem}{Theorem}
\newtheorem{corollary}{Corollary}[theorem]
\newtheorem{lemma}[theorem]{Lemma}

\title{Eigenvector decomposition to determine the existence, shape, and location of numerical oscillations
in Parabolic PDEs}
\author{Corban Harwood and Ethan Jensen}
\date{December 2020}
\begin{document}
\maketitle
\bibliographystyle{ieeetr}
\textbf{Abstract.} In this paper, we employed linear algebra and functional analysis to determine necessary and sufficient conditions for oscillation-free and stable solutions to linear and nonlinear parabolic partial differential equations. We applied singular value decomposition and Fourier analysis to various finite difference schemes to extract patterns in the eigenfunctions (sampled by the eigenvectors) and the shape of their eigenspectrum. With these, we determined how the initial and boundary conditions affect the frequency and long term behavior of numerical oscillations, as well as the location of solution regions most sensitive to them.
\section{Introduction}
Finite difference methods are powerful discretization techniques that approximate the solutions of partial differential equations (PDE) by partitioning the domain into a uniform grid. However, this method of partitioning can lead to severe errors in the numerical solution. Von Neumann\cite{charney1950numerical} developed an analytic tool that relates the stability of the numerical solution to the spacing in time, \(\Delta t\). Similarly for stable solutions, poorly chosen values for \(\Delta t\) can lead to qualitative differences in the numerical solution versus the true solution that manifest as spurious waves\cite{lee1999spurious}, called numerical oscillations. Parabolic PDE are powerful equations that describe many physical phenomena, such as heat flow and particle diffusion. Reaction-diffusion equations in particular, such as the Fisher-KPP equation and the FitzHugh-Nagumo equation have applications in chemistry, and biology\cite{fitzhugh1969mathematical, kolmogorov1937study}. All of these are PDE are affected by numerical oscillations.\cite{teixeira2004stable, john2007spurious} Therefore, it is imperative that we understand the nature of numerical oscillations.
\newline
Previous work by Main and Harwood\cite{harwood2017eigenvalue} has demonstrated the dependence of numerical oscillations on the eigenspectrum of the time-step matrix for two-level finite difference schemes. Motivated by this discovery, this paper calculates a set of conditions for guaranteeing the rapid decay of numerical oscillations in linear diffusion PDE by taking advantage of symmetries in the eigenvectors of various numerical methods. We also show how these conditions may be extended to nonlinear diffusion PDE.

\section{Theory}
\subsection{The Finite Difference Method}
Finite differences are numerical methods that partition the solution into discrete units in time and in space to numerically approximate the true continuous solution. To discretize the PDE, spatial derivatives are approximated with a Taylor series approximation\cite{eberly2008derivative}, and then time derivatives are approximated with a numerical scheme.

\subsection{Special Properties of Toeplitz Matrices}
Toeplitz matrices are important symmetric matrices that contain the finite difference approximations of the second spatial derivative. For example, the second central difference approximation is
\[\frac{\partial^2}{\partial x^2}u(x) = \frac{u(x+\Delta x) - 2u(x) + u(x-\Delta x)}{\Delta x^2} + O(\Delta x^2)\]
Indexing each u, we have
\[(u_i)_{xx} = \frac{1}{\Delta x^2}(u_{i+1} -2u_i + u_{i-1})\]
The corresponding Toeplitz matrix is
\[T_0 = \begin{bmatrix}
-2 & 1 & 0 & 0\\
1 & -2 & \ddots & 0\\
0 & \ddots & \ddots & 1\\
0 & 0 & 1 & -2
\end{bmatrix}\]
Using the tridiagonal matrix algorithm, we calculate the eigenvalues and eigenvectors of \(T_0\) analytically.\cite{hu1996analytical}
Its eigenvalues are defined by
\[\lambda_j = -4\sin^2\left(\frac{\pi j}{2(k+1)}\right),\ j = 1,2,....k\]
The corresponding eigenvectors \(x_j\) are indexed according to their eigenvalues and are defined by sine eigenfunctions. The matrix of eigenvectors S, where \(S = \begin{bmatrix}
x_1 & x_2 & \hdots & x_k
\end{bmatrix}\) is defined by
\[S_{n,j} = \sin\left(\frac{\pi n j}{k+1}\right)\]
This matrix S is one of the three most important matrices in all of Linear Algebra\cite{strang1999discrete}, and it has many interesting properties and symmetries that we make use of in this paper.
\begin{lemma}\label{S symmetry 1}
Let \(S_{n,j} = \sin\left(\frac{\pi n j}{k+1}\right)\). Then,
\[S^{-1} = \left(\frac{2}{k+1}\right)S\]
\end{lemma}
\begin{proof}
S contains all the eigenvectors of \(T_0 \). Since \(T_0\) is symmetric, its eigenvectors are orthogonal, so \(S^TS = D\), for some diagonal matrix \(D\).
\[S_{n,j} = \sin\left(\frac{\pi n j}{k+1}\right)\]
\[\vec{x}_j^T\vec{x}_j = \sum_{n=1}^k \sin^2\left(\frac{\pi j n}{k+1}\right) = \sum_{n=1}^k\frac{1}{2} + \sum_{n=1}^k \cos\left(\frac{2\pi nj}{k+1}\right)= \frac{k+1}{2}\]
Thus, we have
\[S^TS = \left(\frac{k+1}{2}\right)I\]
Since \(S_{n,j} = S_{j,n} = \sin\left(\frac{\pi n j}{k+1}\right)\), \(S\) is also symmetric.
\[\therefore S^{-1} = \left(\frac{2}{k+1}\right)S\]
\end{proof}

\begin{lemma}\label{S symmetry 2}
Let \(S_{n,j} = \sin\left(\frac{\pi n j}{k+1}\right)\). Then,
\[S_{n,j} = (-1)^{n+1}S_{n, k+1-j}\]
\end{lemma}
\begin{proof}
\[S_{n,j} = \sin\left(\frac{\pi n j}{k+1}\right)\]
\[S_{n,k+1-j} = \sin\left(\frac{\pi n (k+1-j)}{k+1}\right) = (-1)^{n+1}\sin\left(\frac{\pi n j}{k+1}\right)\]
\[\therefore S_{n, j} = (-1)^{n+1}S_{n, k+1-j}\]
\end{proof}
\subsection{Numerical Schemes}
Consider the heat equation \(u_t = \delta u_{xx}\). First, approximating in space, we have
\[(u_i)_t = \frac{\delta}{\Delta x^2}(u_{i+1}^n -2u_i^n + u_{i-1}^n)\]
Combining these relations into a relation for the entire solution vector, we have
\[\frac{\partial}{\partial t}(\vec{u}_n) = \frac{\delta}{\Delta x^2}T_0\vec{u}_n\]
\[d(\vec{u_n}) = \frac{\delta \Delta t}{\Delta x^2}T_0\vec{u}_n\]
Since the PDE is linear, the exact solution in time is the matrix exponential
\[\vec{u}_{n+1} = e^{rT_0}\vec{u}_n\]
where \(r = \frac{\delta \Delta t}{\Delta x^2}\) is a simplifying constant. To discretize in time, finite difference approximations of PDE use a numerical scheme, which approximate the matrix exponential solution. Below are a few of the numerical schemes studied.\cite{ascher1997implicit}
\newline
\textbf{Explicit Schemes (Forward Euler, Runge-Kutta)}
\[\vec{u}_{n+1} = (I + rT_0)\vec{u}_n\]
\[\vec{u}_{n+1} = (I + rT_0 + \frac{r^2}{2}T_0^2)\vec{u}_n\]
\[\vec{u}_{n+1} = (I + rT_0 + \frac{r^2}{2}T_0^2 + \frac{r^3}{6}T_0^3)\vec{u}_n\]
\textbf{Implicit Scheme (Backward Euler)}
\[\vec{u}_{n+1} = (I - rT_0)^{-1}\vec{u}_n\]
\textbf{Semi-Implicit Scheme (Crank-Nicolson)}\cite{crank1947practical}
\[\vec{u}_{n+1} = (I - \frac{r}{2}T_0)^{-1}(I + \frac{r}{2}T_0)\vec{u}_n\]
If we consider a general scheme \(\vec{u}_{n+1} = M\vec{u}_n + \vec{B}\) where M is the time-step matrix, then there exist polynomials p and q such that \(M = q^{-1}(rT_0)p(rT_0)\) because each numerical scheme is a Taylor approximation to the matrix exponential. Thus, all time-step matrices for the heat equation can derive their eigenvectors and eigenvalues from \(T_0\). See Theorem \ref{eigenpairs of polynomials} in the Appendix.
\section{Methodology}
Numerical oscillations can only occur when the eigenvalues of the time-step matrix start to become negative. Thus, numerical oscillations in the solution can be prevented by restricting the eigenspectrum of the time-step matrices to non-negative values (\textit{Non-negative Eigenvalue Condition}).
\begin{theorem}\label{FTNO}[Fundamental Theorem of Numerical Oscillations]
Let \(u_t = \delta u_{xx}\) be given by an approximation \(u^{n+1} = Mu^{n}\) and let M = \(R(T_0)\) be a rational function of \(T_0\). Then, numerical oscillations are absent if and only if
\[|a_j|R(r\lambda_j) \geq 0,\ \ j = 1,2,3,...k\]
where \(a_j\) are the Fourier coefficients from the discrete finite sine transform of \(u_0 - \overline{u}\)  (initial condition minus the steady state), and \(R(r\lambda_j)\) are the eigenvalues of M.
\end{theorem}
\begin{proof}
All finite difference schemes of the heat equation are of the form
\[\vec{u}_n = M\vec{u}_{n-1} + \vec{B}\]
where \(M = R(rT_0) = q^{-1}(rT_0)p(rT_0)\) and \(\vec{B} = q^{-1}(rT_0)\vec{b}\), where \(\vec{b}\) is a sparse vector containing the boundary conditions.
\[\vec{u}_n = M^nu_0 + \left(\sum_{k=0}^{n-1}M^k\right)\vec{B}\]
\[\vec{u}_n = M^nu_0 + (M^n-I)(M-I)^{-1}\vec{B}\]
\begin{equation}\label{expanded form}
    \vec{u}_n = M^n(u_0 - (I-M)^{-1}\vec{B}) + (I-M)^{-1}\vec{B}
\end{equation}
\[\lim_{n\rightarrow \infty}\vec{u}_n = \lim_{n\rightarrow \infty}M^n(u_0 - (I-M)^{-1}\vec{B}) + (I-M)^{-1}\vec{B}\]
Since the scheme is stable \(\lim_{n\rightarrow \infty}M^n = 0\). Thus,
\[\overline{u} = (I-M)^{-1}\vec{B}\]
where \(\overline{u}\) is the steady state. Plugging this back into \eqref{expanded form}, we have
\[\vec{u}_n = M^n(u_0 - \overline{u}) + \overline{u}\]
Performing the eigenvector decomposition, we have
\[\vec{u}_n = STS^{-1}(u_0 - \overline{u}) + \overline{u}\]
From Theorem \ref{S symmetry 1}, we have
\[\vec{u}_n = \frac{2}{k+1}STS(u_0 - \overline{u}) + \overline{u}\]
From Theorem \ref{eigenpairs of polynomials}, we have
\[\vec{u}_n = \frac{2}{k+1}\sum_{j=1}^ka_jR(r\lambda_j)^n\vec{x}_j + \overline{u}\]
Where \(a_j\) are the Fourier coefficients from the discrete finite sine transform of \(u_0 - \overline{u}\), \(\lambda_j, \vec{x}_j\) are the eigenvalues and eigenvectors of \(T_0\).
\newline
Thus, if \(R(r\lambda_j) < 0\), solution components will oscillate as they are being powered up unless \(|a_j|=0\). Therefore, numerical oscillations are absent if and only if
\[|a_j|R(r\lambda_J) \geq 0,\ \ j = 1,2,3...k\]
\end{proof}
\noindent
Therefore, numerical oscillations are not just dependant on the numerical scheme. The initial strength of the numerical oscillations is also dependant on the smoothness of the initial condition and boundary conditions.
\subsection{Eigenspectrums of numerical methods}
The eigenspectrum of explicit methods for the heat equation take the form \[(\lambda_M)_j = R(r\lambda_j) = 1 + r\lambda_j + \frac{r^2}{2}\lambda_j^2 + \frac{r^3}{6}\lambda_j^3\]
The eigenspectrum of the Crank-Nicolson method for the heat equation is
\[(\lambda_M)_j = R(r\lambda_j) = \frac{1+(r/2)\lambda_j}{1-(r/2)\lambda_j}\]
Even-order Runge-kutta methods as well as the implicit backward Euler method have a strictly positive eigenspectrum - these schemes do not produce numerical oscillations.
\subsection{Shape of Numerical Oscillations}
Odd order Runge-Kutta methods, such as the forward Euler method, are increasing functions. Likewise, the Crank-Nicolson eigenspectrum is also an increasing function on \((-4,0)\). Since \(\lambda_j = -4\sin^2\left(\frac{\pi j}{2(k+1)}\right)\) is a decreasing function in j, the eigenvalues of the odd-order Runge-Kutta methods and the Crank-Nicolson method are decreasing in j.
\newline
\newline
Thus, of the schemes that can produce numerical oscillations, the smallest eigenvalues get paired with the highest frequency eigenfunction, and the frequency of the corresponding eigenfunction gets smaller as the eigenvalue gets bigger. Hence, if the eigenspectrum is partially negative, then the eigenvalues that are negative will be indexed with the highest frequency eigenvectors. This explains why numerical oscillations usually manifest as a high frequency phenomenon in linear parabolic PDE.
\section{Results}
\subsection{Balanced Eigenvalue Condition}
From Theorem \ref{FTNO}, the non-negative eigenvalue condition,
\[|a_j|R(r\lambda_j) \geq 0,\ j = 1,2,3,...k\]
ensures that no numerical oscillations occur in the solution. Beyond this, however, is a secondary condition that guarantees fast-decaying numerical oscillations, almost imperceptible relative to the numerical solution. We call this the \textit{balanced eigenvalue condition}, given by
\[R(r\lambda_j) + R(r\lambda_{k+1-j}) \geq 0,\ \ |a_j| \geq |a_{k+1-j}|\ j = 1,2,3,...\left\lfloor\frac{k+1}{2}\right\rfloor.\]
The part \(R(r\lambda_j) + R(r\lambda_{k+1-j}) \geq 0\) ensures that the eigenfunction indexed by \(k+1-j\) will decay faster than the corresponding envelope eigenfunction indexed by \(j\). The part \(|a_j| \geq |a_{k+1-j}|\) ensures that the initial amplitude of the oscillating wave will start smaller than the amplitude of the corresponding envelope.
\newline
\newline
The balanced eigenvalue condition is motivated by lemma \ref{S symmetry 2}. Lemma 2 shows that eigenfunctions indexed by \(k+1-j\) have a corresponding eigenfunction indexed by \(j\) which acts as an envelope, matching the magnitude of the eigenvector at every index and equalling the eigenfunction at every odd index. These relations from Lemma 2 are visualized in Figure 1.
\newline
\begin{figure}[h!]
    \centering
    \includegraphics[width=12cm]{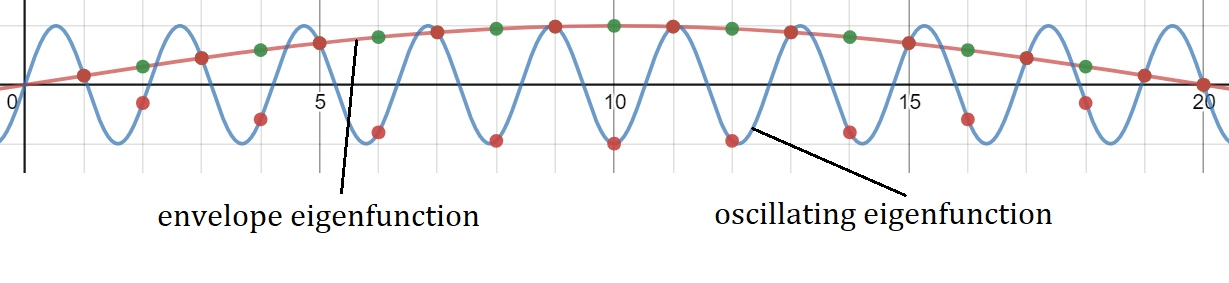}
    \caption{Comparison of Matching Eigenfunctions}
\end{figure}
\newline
The initial conditions that satisfy \(|a_j| \geq |a_{k+1-j}|\) are smooth functions with little total variation, for example, a monotonic function\cite{booton2015general}. Initial conditions that are not smooth incite numerical oscillations by powering up the eigenvalues paired with the high frequency eigenvectors. Disagreements between the initial condition and the boundary condition cause numerical oscillations because large changes at the smallest level of resolution powers up the eigenvalue of the highest frequency eigenvector.
\newline
\newline
The free parameters in the setup of the numerical solution are combined into the simplifying constant \(r = \delta \Delta t/\Delta x^2\). To get conditions for r that predict numerical oscillations, we calculate the largest value of r such that the non-negative eigenvalue condition and the balanced eigenvalue condition are true for each numerical scheme.
\newline
\newline
Consider the heat equation
\[u_t = \delta u_{xx}\]
The time-step matrix for the heat equation is a rational function of \(rT_0\). Thus, the eigenvalues are rational functions of \(r\lambda_j\) and the non-negative and balanced eigenvalue conditions for the heat equation can be solved in terms of \(r\lambda_j\). Then, using the fact \(\textup{min}(\lambda_j) = -4\) obtains appropriate bounds for r.
\newline
\newline
Explicit methods are conditionally stable. For the odd-order Runge-Kutta methods, the balanced eigenvalue condition is met if the solution is stable.\cite{moler2004numerical} In contrast, the Crank-Nicolson method on the heat equation is unconditionally stable. r can be increased indefinitely, but as it is increased the balanced eigenvalue condition is violated and numerical oscillations become both more prevalent and take longer to decay.
\newline
\newline
Consider the linear reaction diffusion equation
\[u_t = \delta u_{xx} - \rho u\]
with spatial discretization
\[(\vec{u}_n)_t = \left(\frac{\delta}{\Delta x^2}T_0 - \rho I\right)\vec{u}_n\] The time-step matrix of the linear reaction diffusion equation is also rational function of \(rT_0 - \rho\Delta t I\). Thus, the eigenvectors remain the same as those of the heat equation, and we can use the same kind of analysis. However, the change in the eigenvalues means that the non-negative eigenvalue condition and the balanced eigenvalue condition need to be recalculated. See Table 1 for the Non-negative Eigenvalue conditions and balanced eigenvalue conditions for each numerical scheme.
\newline
\newline
These oscillatory conditions are shown graphically in Figure 2.
\newline
\begin{figure}[h!]
    \centering
    \includegraphics[width = 12 cm]{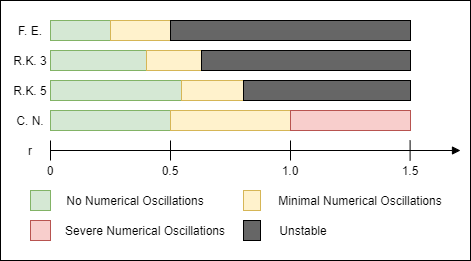}
    \caption{Oscillatory Conditions}
\end{figure}
\newline
\subsection{Demonstration}
We demonstrate the effect of the initial conditions and boundary conditions have on a numerical solution whose eigenspectrum satisfies the balanced eigenvalue condition. In Figure 3 we show color maps of various numerical solutions to the heat equation using the forward Euler method, with \(r = 0.5\).
\newline
\begin{figure}[h!]
    \centering
    \includegraphics[width = 9cm]{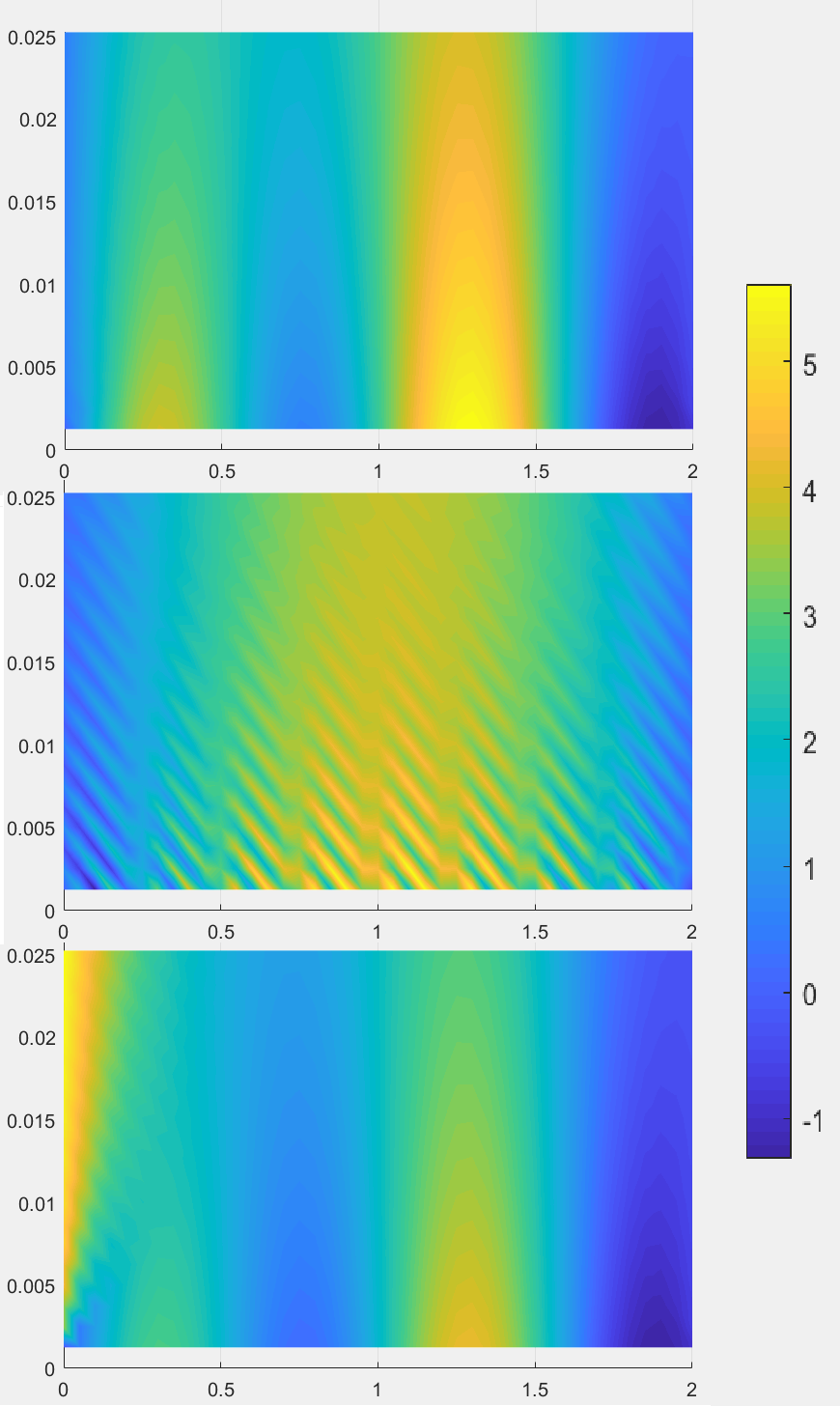}
    \caption{Simulations of the heat equation}
    \label{fig:my_label}
\end{figure}
\newline
The top image depicts a solution to the heat equation with a very smooth initial condition, and agreeing boundary conditions. The numerical scheme satisfies the balanced eigenvalue condition. Hence, minimal numerical oscillations are present. The middle image depicts a solution to the heat equation with a very noisy initial condition. Despite satisfying the balanced eigenvalue condition this initial condition incited numerical oscillations that match the frequency of the high-frequency components in the initial condition which are seen in the figure as stripes. The bottom image depicts a solution to the heat equation whose initial condition is smooth, but whose left-hand boundary condition did not match the initial condition, creating very slight numerical oscillations on the left edge.
\subsection{Nonlinear Reaction-Diffusion Equations}
In addition to linear PDE, we also studied various nonlinear PDE.
\newline
\newline
Nonlinear Reaction-Diffusion: \(u_t = \delta u_{xx} - \rho u^2\)
\newline
Fisher-KPP: \(u_t = \delta u_{xx} - \rho u(1-u)\)
\newline
FitzHugh-Nagumo Reaction-Diffusion: \(u_t = \delta u_{xx} - \rho u(1-u)(a-u)\)
\newline
\newline
Each of these PDE has at least one stable steady state solution, a curve satisfying the boundary conditions that is drawn towards some constant function \(u_0\). Table 1 shows the linearizations of the three nonlinear PDE about their corresponding steady states. Since the linearizations of the nonlinear PDE match the form of the linear diffusion PDE, the non-negative eigenvalue condition of the linearizations are a worst-case bounding to prevent numerical oscillations (see corollary \ref{cor: nonlinear non-negative eigenvalue condition} in the appendix).
\subsection{Location of Numerical Oscillations}
Combining the nonlinear portion with the diffusion term produces orthogonal eigenfunctions\cite{osipov2017evaluation} with shapes that resemble functions balanced between pure sine waves and elementary basis vectors. Figure 4 shows the eigenvectors of the time-step matrix of the Forward Euler method on the Fisher-KPP equation about a monotone steady state, ordered by eigenvalue, with r=1 and \(\rho/\delta = 1\). The eigenvectors with the smallest eigenvalues retain much of their wave-like form at the end of the solution closest to the stable steady state and localising the numerical oscillations to the portion of the solution nearest to the stable equilibrium.
\subsection{Application of the Balanced Eigenvalue Condition}
Our work also showed that a symmetry in the eigenvectors of nonlinear schemes, similar to that from lemma \ref{S symmetry 2} exists for these eigenfunctions, meaning that the balanced eigenvalue condition can be partially applied for numerical methods of nonlinear PDE. In Figure 5, eigenvectors from the time-step matrix of the Fisher-KPP equation are shown, split between those that are more tightly centered (shown in gray), and those that maintain wave-like properties (shown in color). Corresponding eigenfunctions having components that closely match each other in magnitude are shown with the same color.
\newline
\begin{figure}[h!]
    \centering
    \includegraphics[width = 10cm]{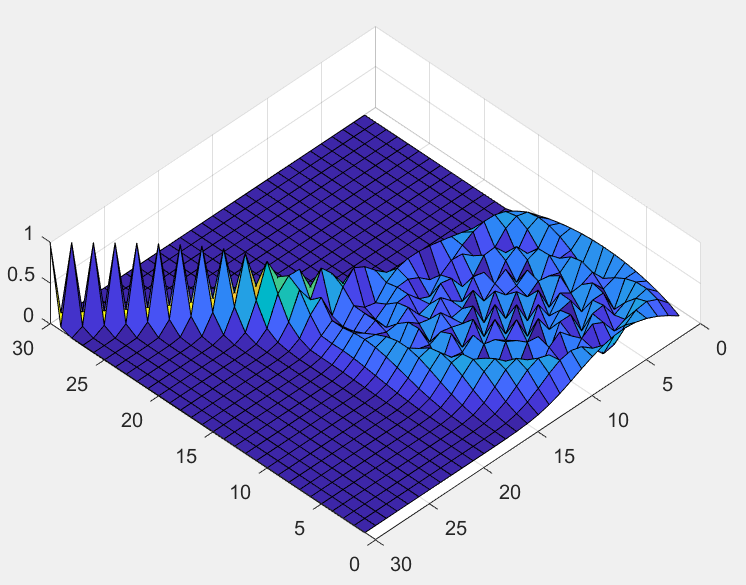}
    \caption{Eigenvectors for the Fisher-KPP equation}
    \label{fig:my_label}
\end{figure}
\newline
\begin{figure}[h!]
    \centering
    \includegraphics[width = 12cm]{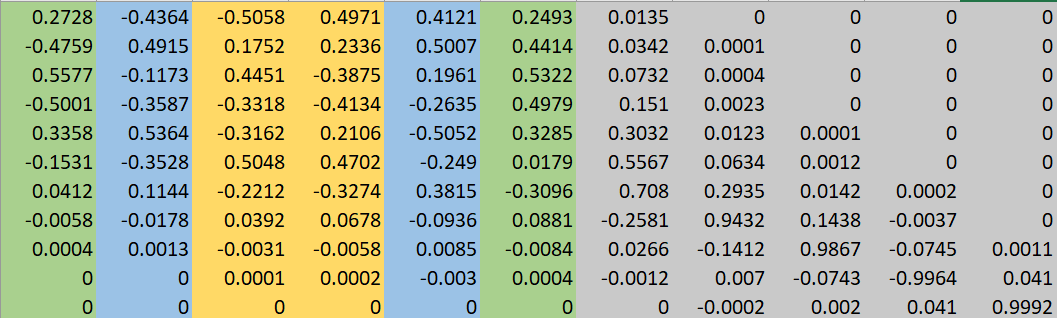}
    \caption{Nonlinear Eigenfunction Symmetry}
    \label{fig:nonlinear eigenvector symmetry}
\end{figure}
\newpage
\section{Discussion}
\subsection{Characterization of Numerical Oscillations}
In this paper, Fourier analysis was used to give a description of the numerical solution as a sum of waves. Numerical oscillations, in turn were characterised by the high frequency sine eigenfunctions. Finding patterns in these eigenfunctions established a carrier wave relationship between high-frequency eigenfunctions and their opposing low-frequency eigenfunctions. The Fourier analysis showed that the balanced eigenvalue condition which results from this pairing guarantees that the numerical oscillations decay fast relative to the solution and that initial conditions can be chosen to either create or avoid numerical oscillations.
\subsection{Analysis of Numerical Oscillations in Nonlinear PDE}
In addition to linear PDE, several nonlinear reaction-diffusion PDE were also studied. By looking at the method matrices of the numerical methods about stable steady-states, the structure of the eigenfunctions of the nonlinear equations could be deduced. The nonlinear component, mixed with the diffusion component creates orthogonal eigenfunctions which resemble sine functions when the solution gets close to a stable equilibrium. This localizes the numerical oscillations to the part of the function closest to the stable equilibrium. Additionally, the rough symmetry of these wave-like components suggests that the balanced eigenvalue condition could also be applied to the numerical methods of nonlinear reaction diffusion equations.
\subsection{Final Remarks and Further Research}
In contrast to von Neumann stability analysis, this paper makes use of the precise values of the eigenvalues of the numerical methods to make claims about the quality of the numerical solution. The same method of analysis can likely be generalized to other numerical techniques that rely on matrix representations of the solution. Finite volume and finite element methods for example, although using different methods to discretize in space, all make use of a time-step matrix to approximate PDE in time.
\newpage
\bibliography{References}
\newpage
\section{Appendix}
\begin{theorem}\label{eigenpairs of polynomials}
Let p and q be polynomials, A and B be square matrices, and B invertible, and \(A = p(B)q^{-1}(B)\). Then for all eigenvalues \(\lambda_B\) of B, \(\lambda_A = p(\lambda_B)q^{-1}(\lambda_B)\) is an eigenvalue of A, and all eigenvectors \(x\) of B are eigenvectors of A.
\end{theorem}
\begin{proof}
We start the proof with mathematical induction.
Let \(\lambda_B\) be an eigenvalue of \(B\) and \(x\) be its corresponding eigenvector. Then our base case is
\begin{equation}
    Bx = \lambda_Bx
\end{equation}
Let us assume the following
\begin{equation}
    B^k x = \lambda_B^k x
\end{equation}
Then
\begin{equation}
    B^{k+1} x = BB^k x = B\lambda_B^k x = \lambda_B^kBx = \lambda_B^{k+1}x
\end{equation}
So \(B^nx = \lambda_B^n x\) for \(n = 1,2,3,...\) by mathematical induction.
\newline
Let \(p(x) = \sum_{n=0}^{\infty}a_nx^n\) and \(q(x) = \sum_{n=0}^{\infty}b_nx^n\) be polynomial functions, and let \(A = p(B)q^{-1}(B)\)
\begin{equation}
    Aq(B) = p(B)
\end{equation}
\begin{equation}
    Aq(B)x = p(B)x
\end{equation}
\begin{equation}
    A\sum_{n=0}^{\infty}b_nB^nx = \sum_{n=0}^{\infty}a_nB^nx
\end{equation}
\begin{equation}
    A\sum_{n=0}^{\infty}b_n\lambda_B^nx = \sum_{n=0}^{\infty}a_n\lambda_B^nx
\end{equation}
\begin{equation}
    Aq(\lambda_B)x = p(\lambda_B)x
\end{equation}
Multiplication is associative with constants.
\begin{equation}
    Ax = p(\lambda_B)q^{-1}(\lambda_B)x
\end{equation}
Thus, x is also an eigenvector of A, and its eigenvalue is \(p(\lambda_B)q^{-1}(\lambda_B)\).
\end{proof}

\begin{lemma}\label{thm:pos def}
    Let \(A\) be positive definite. If \(B-A\) is a positive definite, then \(B\) is positive definite.
\end{lemma}
\begin{proof}
    Let \(A\) and \(B-A\) be positive definite.
    \[x^TAx \geq 0 \ \forall x\]
    \[x^T(B-A)x \geq 0 \ \forall x\]
    \[x^TBx - x^TAx \geq 0 \ \forall x\]
    \[\therefore x^TBx \geq 0 \ \forall x\]
    So B is positive definite.
\end{proof}

\begin{theorem}\label{thm: pos d}
    Let \(M - N(\vec{u}_n)\) be the time-step matrix of a nonlinear diffusion PDE where M is the time-step matrix for the diffusion part, and N is a nonlinear vector function. If \(N(\hat{u}) - N(\vec{u}_n)\) is positive definite for all n for some constant \(\hat{u}\), and \(M - N(\hat{u})\) is positive definite, then \(M - N(\vec{u}_n)\) is positive definite for all n.
\end{theorem}
\begin{proof}
    \[[M - N(\vec{u}_n)] - [M - N(\hat{u})] = N(\hat{u}) - N(\vec{u}_n)\]
    \(M - N(\hat{u})\) and \(N(\hat{u}) - N(\vec{u}_n)\) are both positive definite for all n.
    \newline
    \(\therefore M - N(\vec{u}_n)\) is also positive definite for all n by lemma \ref{thm:pos def}
\end{proof}

\begin{corollary}\label{cor: nonlinear non-negative eigenvalue condition}
    Nonnegative eigenvalues for numerical methods of the equations \(u_t = u_{xx} - \rho u^2\), \(u_t = u_{xx} - \rho u(1-u)\), \(u_t = u_{xx} - \rho u(1-u)(a-u)\) is guaranteed by the nonnegative eigenvalue condition about the linearizations of those schemes about their respective stable equilibria.
\end{corollary}
\begin{proof}
    \textbf{(1)} \(u_t = u_{xx} - \rho u^2\)
    \newline
    Discretization: \(\vec{u}_{n+1} = (M - \rho \cdot diag(\vec{u}_n))\vec{u}_n\)
    \newline
    Discretization of the linearization: \(\vec{u}_{n+1} = M\vec{u}_n\)
    \[M - (M - \rho \cdot diag(\vec{u}_n)) = \rho \cdot diag(\vec{u}_n)\]
    Thus, if the time-step matrix M of the linearization is positive definite, then the time-step matrix \(M - diag(\vec{u}_n)\) is positive definite for \(u(n, i) \geq 0\) by theorem \ref{thm: pos d}.
    \newline
    \newline
    \textbf{(2)} \(u_t = u_{xx} - \rho u(1-u)\)
    \newline
    Discretization: \(\vec{u}_{n+1} = (M - \rho(I - diag(\vec{u}_n)))\vec{u}_n\)
    \newline
    Discretization of the linearization: \(\vec{u}_{n+1} = (M - \rho I)\vec{u}_n\)
    \[(M - \rho I) - (M - \rho(I - diag(\vec{u}_n))) = diag(\vec{u}_n)\]
    Thus, if the time-step matrix \(M - \rho I\) of the linearization is positive definite, then the time-step matrix \(M - \rho(I - diag(\vec{u}_n))\) is positive definite for \(u(n,i) \geq 0\) by theorem \ref{thm: pos d}.
    \newline
    \newline
    \textbf{(3)} \(u_t = u_{xx} - \rho u(1-u)(a-u)\)
    \newline
    Discretization: \(\vec{u}_{n+1} = (M - \rho(I - diag(\vec{u}_n)(aI - diag(\vec{u}_n))))\vec{u}_n\)
    \newline
    Discretization of the linearization (about \(u_0 = 0\)): \(\vec{u}_{n+1} = (M + \rho aI)\vec{u}_n\)
    \newline
    Discretization of the linearization (about \(u_0 = 1\)): \(\vec{u}_{n+1} = (M + \rho (1-a)I)\vec{u}_n\)
    \[(M + \rho aI) - (M - \rho(I - diag(\vec{u}_n)(aI - diag(\vec{u}_n)))) = \rho(2aI -(1+a)diag(\vec{u}_n) + diag(\vec{u}_n)^2)\]
    Since \(\rho(2a -(1+a)x + x^2) \geq 0\) for \(0 \leq a \leq 1,\ 0 < x < 1\), \(\rho(2aI -(1+a)diag(\vec{u}_n) + diag(\vec{u}_n)^2)\) is positive definite by theorem \ref{thm: pos d}.
    \[(M + \rho (1-a)I) - (M - \rho(I - diag(\vec{u}_n)(aI - diag(\vec{u}_n)))) = \rho(I -(1+a)diag(\vec{u}_n) + diag(\vec{u}_n)^2)\]
    Since \(\rho(1-(1+a)x + x^2) \geq 0\) for \(0.1 \leq a \leq 1,\ 0 < x < 1\), \(\rho(I -(1+a)diag(\vec{u}_n) + diag(\vec{u}_n)^2)\) is positive definite by theorem \ref{thm: pos d}.
\end{proof}
\newpage
\section{Tables}
\begin{table}[ht]
\centering
\begin{tabular}{|c|c|c|}
    \hline
    \multicolumn{3}{|c|}{Heat equation \(u_t = \delta u_{xx}\)}
    \\
    \hline
    Numerical Scheme & N.N eig. cond. & Bal. eig. cond. \\
    \hline
    Forward Euler & \(r \leq 1/4\) & \(r \leq 1/2\)
    \\
    \hline
    Runge-Kutta 3 & \(r \leq 0.39902\) & \(r \leq 0.62819\)
    \\
    \hline
    Runge-Kutta 5 & \(r \leq 0.54515\) & \(r \leq 0.80426\)
    \\
    \hline
    Crank Nicolson & \(r \leq 1/2\) & \(r \leq 1\)
    \\
    \hline
    \multicolumn{3}{|c|}{Linear Reaction-Diffusion \(u_t = \delta u_{xx} - \rho u\)}
    \\
    \hline
    Numerical Scheme & N.N eig. cond. & Bal. eig. cond. \\
    \hline
    Forward Euler &\(r \leq  1/4 - 1/4\rho \Delta t\) & \(r \leq 1/2 - 1/2\rho \Delta t\)
    \\
    \hline
    Crank Nicolson & \(r \leq 1/2 - 1/4\rho \Delta t\) & \(r \leq 1 - 1/2\rho \Delta t\)
    \\
    \hline
\end{tabular}
\caption{Non-oscillatiory conditions}
\end{table}
\begin{table}[h!]
    \centering
    \begin{tabular}{|c|c|c|}
\hline
Nonlinear Equation & Steady State & Linearization
\\
\hline
\(u_t = \delta u_{xx} - \rho u^2\) & \(u_0 = 0\) & \(u_t = \delta u_{xx}\)
\\
\hline
\(u_t = \delta u_{xx} - \rho u(1-u)\) & \(u_0 = 1\) & \(u_t = \delta u_{xx} - \rho u\)
\\
\hline
\(u_t = \delta u_{xx} - \rho u(1-u)(a-u)\) & \(u_0 = 0\) & \(u_t = \delta u{xx} + \rho au\)
\\
\hline
\(u_t = \delta u_{xx} - \rho u(1-u)(a-u)\) & \(u_0 = 1\) & \(u_t = \delta u{xx} + \rho (1-a)u\)
\\
\hline
\end{tabular}
\caption{Linearizations of Nonlinear Reaction-diffusion Equations}
\end{table}
\end{document}